\documentclass[11pt]{amsart}
\usepackage{amsmath,amsthm,amsfonts,amssymb,times}
\usepackage{latexsym,epsfig,subfigure,afterpage}
\usepackage{hhline}
\usepackage{array}
\usepackage{tabularx}
\usepackage{longtable}

\DeclareMathAlphabet{\curly}{U}{rsfs}{m}{n}

\theoremstyle{remark}

\theoremstyle{plain}

\newtheorem{lem}{Lemma}[section]
\newtheorem{thm}{Theorem}

\numberwithin{equation}{section}

\newcommand{\li}{{\rm li}}

\newcommand{\be}{\begin{equation}}
\newcommand{\ee}{\end{equation}}
\newcommand{\benn}{\begin{equation*}}
\newcommand{\eenn}{\end{equation*}}
\newcommand{\bal}{\begin{align*}}
\newcommand{\ea}{\end{align*}}
\newcommand{\eal}{\ensuremath{\end{align*}}}
\newcommand{\bea}{\begin{eqnarray}}
\newcommand{\eea}{\end{eqnarray}}

\newcommand{\g}{\ensuremath{\gamma}}

\newcommand{\del}{\ensuremath{\delta}}
\newcommand{\eps}{\ensuremath{\varepsilon}}
\renewcommand{\(}{\left(}
\renewcommand{\)}{\right)}

\newcommand{\pfrac}[2]{\left(\frac{#1}{#2}\right)}

\renewcommand{\le}{\leqslant}
\renewcommand{\leq}{\leqslant}
\renewcommand{\ge}{\geqslant}
\renewcommand{\geq}{\geqslant}

\begin{document}

\title{Chebyshev's bias for products of two primes}

\author{Kevin Ford}
\email{ford@math.uiuc.edu, jpsneed@uiuc.edu}
\author{Jason Sneed}
\address{
Department of Mathematics,
University of Illinois at Urbana-Champaign,
1409 West Green St., Urbana, IL 61801}

\date{\today}
\begin{abstract} Under two assumptions, 
we determine the distribution of the difference between two
functions each counting the numbers $\le x$ that are in 
a given arithmetic progression modulo $q$ and the product of two primes.
 The two assumptions are (i) the Extended Riemann Hypothesis for 
Dirichlet $L$-functions modulo $q$, and (ii) that
the imaginary parts of the nontrivial zeros of these $L$-functions are 
linearly independent over the rationals.  Our results are analogs of
similar results proved for primes in arithmetic progressions 
by Rubinstein and Sarnak.
\end{abstract}

\thanks{2000 Mathematics Subject Classification:11M06, 11N13, 11N25}
\thanks{The research of K.~F. was supported in part
by National Science Foundation grants DMS-0555367 and DMS-0901339.}

\maketitle

%%%%%%%%%%%%%%%%%%%%%%%%%%%

\section{Introduction}

%%%%%%%%%%%%%%%%%%%%%%%%%%%

\subsection{Prime number races}
Let $\pi(x;q,a)$ denote the number of primes in the progression $a\!\! \mod q$.
For fixed $q$, the functions $\pi(x;q,a)$ 
(for $a\in A_q$, the set of residues coprime to $q$) 
all satisfy
\begin{equation}\label{PNT}  \pi(x,q,a) \sim \frac{x}{\varphi(q)\log x},
\end{equation}
where $\varphi$ is Euler's totient function \cite{Da}.
There are, however, curious inequities.  For example
$\pi(x;4,3) \ge \pi(x;4,1)$ seems to hold for most $x$,
an observation of Chebyshev from 1853 \cite{Ch}. 
In fact, $\pi(x;4,3) < \pi(x;4,1)$ for the first time
at $x = 26,861$ \cite{Le}. 
More generally, one can ask various questions about the behavior of
\be\label{Delta}
\Delta(x;q,a,b) := \pi(x;q,a)-\pi(x;q,b)
\ee
for distinct $a,b\in A_q$.  Does $\Delta(x;q,a,b)$ change sign
infinitely often?  Where is the first sign change?  How many sign
changes with $x\le X$ ?  What are the extreme values of
$\Delta(x;q,a,b)$?  Such questions are colloquially known 
as \emph{prime race problems}, and were studied extensively by
Knapowski and Tur\'an in a series of papers beginning with \cite{KT}.
 See the survey articles
\cite{FK} and \cite{GM} and references therein for an introduction 
to the subject and summary of major findings.  Properties
of Dirichlet $L$-functions lie at the heart of such investigations.

Despite the tendency for the function  $\Delta(x;4,3,1)$ to be negative,
Littlewood \cite{Li} showed that it changes sign
infinitely often.   Similar results have been proved for other
$q,a,b$ (see \cite{JS} and references therein).
Still, in light of Chebyshev's observation, we can ask how frequently
$\Delta(x;q,a,b)$ is positive and how often it is negative.
These questions are best addressed in the context of \emph{logarithmic
density}.  A set $S$ of positive integers has logarithmic density
$$
\del(S) =  \lim_{x\to\infty} \frac{1}{\log x}
\sum_{\substack{n\le x \\ n\in S}} \frac{1}{n}
$$
provided the limit exists.  Let $\del(q,a,b) = \del(P(q,a,b))$, where
$P(q,a,b)$ is the set of integers $n$ with $\Delta(n;q,a,b)>0$.
In 1994, Rubinstein and Sarnak \cite{RS}  showed that
$\delta(q;a,b)$ exists, assuming two
hypotheses (i) the Extended Riemann Hypothesis for Dirichlet
$L$-functions modulo $q$ (ERH$_q$), and (ii) the imaginary parts of
zeros of each Dirichlet $L$-function are linearly independent over the
rationals (GSH$_q$ - Grand Simplicity Hypothesis).
The authors also gave methods to accurately estimate the ``bias'',
for example showing that
$\delta(4;3,1)\approx 0.996$ in Chebyshev's case.
More generally, $\delta(q;a,b)=\frac12$ when $a$ and $b$ are either
both quadratic
residues modulo $q$ or both quadratic nonresidues (unbiased prime races), but
$\delta(q;a,b)>\frac12$ whenever $a$ is a quadratic
non-residue and $b$ is a quadratic residue.  A bit later we will
discuss the reasons behind these phenomena.  Sharp asymptotics for
$\delta(q;a,b)$ have recently been given by Fiorilli and Martin
\cite{FM}, which explain other properties of these densities.

%%%%%%%%%%%%%%%%%%%%%%%%%%%%%%%%%%%%%%%%%%
%
\subsection{Quasi-prime races}
%
%%%%

In this paper we develop a parallel theory for comparison of functions
$\pi_2(x;q,a)$, the number of integers $\le x$ which are in the
progression $a\!\! \mod q$ and which are the product
of two primes $p_1p_2$ ($p_1=p_2$ allowed).
Put
$$
\Delta_2(x;q,a,b) := \pi_2(x;q,a)-\pi_2(x;q,b),
$$ 
let $P_2(q,a,b)$ be the set of integers $n$ with $\Delta_2(n;q,a,b)>0$, and set $\del_2(q,a,b)=\del(P_2(q,a,b))$.
The table below shows all such
quasi-primes up to 100 grouped in residue classes modulo 4.

\vspace{.1in}
\begin{center}
\begin{tabular}{|c|c|}
\hline $pq \equiv 1 \pmod 4$ & $pq \equiv 3 \pmod 4$ \\
\hline 9 & 15\\
\hline 21 & 35\\
\hline 25 & 39\\
\hline 33 & 51\\
\hline 49 & 55\\
\hline 57 & 87\\
\hline 65 & 91\\
\hline 69 & 95\\
\hline 77 & \\
\hline 85 & \\
\hline 93 & \\
\hline
\end{tabular}
\end{center}
\vspace{.1in}

Observe that $\Delta_2(x;4,3,1) \le 0$ for $x\le 100$, and in fact
the smallest $x$ with $\Delta_2(x;4,3,1) > 0$ is $x=26747$ (amazingly
close to the first sign change of $\Delta(x;4,3,1)$).
Some years ago Richard Hudson conjectured that the bias for products
of two primes is always reversed from that of primes; i.e.,
$\delta_2(q;a,b)<\frac12$ when $a$ is a quadratic non-residue modulo
$q$ and $b$ is a quadratic residue.
Under the same assumptions as
\cite{RS}, namely ERH$_q$ and GSH$_q$, we confirm Hudson's conjecture 
and also show that the bias is less pronounced. 

\begin{thm}\label{thm1}
Let $a,b$ be distinct elements of $A_q$.  Assuming ERH$_q$ and GSH$_q$,
$\delta_2(q;a,b)$ exists.  Moreover, if $a$ and $b$ are both quadratic
residues modulo $q$ or both quadratic non-residues, then
$\delta_2(q;a,b)=\frac12$.  Otherwise, if $a$ is a quadratic
nonresidue and $b$ is a quadratic residue, then
$$ 
1 - \delta(q;a,b) < \delta_2(q;a,b) < \frac12.
$$
\end{thm}

We can accurately estimate $\delta_2(q;a,b)$ borrowing methods from
\cite[\S 4]{RS}.  In particular we have
$$
\delta_2(4;3,1) \approx 0.10572.
$$
We deduce Theorem \ref{thm1} by connecting the distribution of 
$\Delta_2(x;q,a,b)$ with the distribution of $\Delta(x;q,a,b)$.
Although the relationship is ``simple'', there is no elementary way
to derive it, say by writing 
$$
\pi_2(x;q,a) = \frac12 \sum_{p\le x} \pi\(\frac{x}{p};q,ap^{-1}\!\! \mod q\)
+ \frac12 \sum_{\substack{p\le \sqrt{x} \\ p^2\equiv a\!\!\! \pmod{q}}} 1.
$$
In particular, our result depends strongly on the assumption that the
zeros of the $L$-functions modulo $q$ have only simple zeros.
Let $N(q,a)$ be the number of $x\in A_q$ with $x^2\equiv a\pmod{q}$,
and let $C(q)$ be the set of nonprincipal Dirichlet characters modulo $q$.

\begin{thm}\label{DeltaDelta2}
Assume ERH$_q$ and for each $\chi\in C(q)$, $L(\frac12,\chi)\ne 0$ and
the zeros of $L(s,\chi)$ are simple.  Then
$$
\frac{\Delta_2(x;q,a,b) \log x}{\sqrt{x} \log\log x} = \frac{N(q,b)-N(q,a)}{2\phi(q)} -
\frac{\log x}{\sqrt{x}} \Delta(x;q,a,b) + \Sigma(x;q,a,b),
$$
where $\frac{1}{Y} \int_1^Y |\Sigma(e^y;q,a,b)|^2\, dy = o(1)$ as $Y\to \infty$.
\end{thm}

The expression for $\Delta_2$ given in Theorem \ref{DeltaDelta2} must be modified if some $L(s,\chi)$ has multiple zeros; 
see \S \ref{sec:outline} for details.

Figures 1,2 and 3 show
 graphs corresponding to $(q,a,b)=(4,3,1)$, plotted on a
logarithmic scale from $x=10^3$ to $x=10^9$.
While $\Sigma(x;4,3,1)$ appears to be oscillating around $-0.2$, this is 
caused by some terms in $\Sigma(x;4,3,1)$ of order $1/\log\log x$, and
$\log\log 10^9 \approx 3.03$.  By Theorem \ref{DeltaDelta2}, 
$\Sigma(x;4,3,1)$ will (assuming ERH$_4$ and GSH$_4$)
 eventually settle down to oscillating about 0.

\begin{figure}
\input{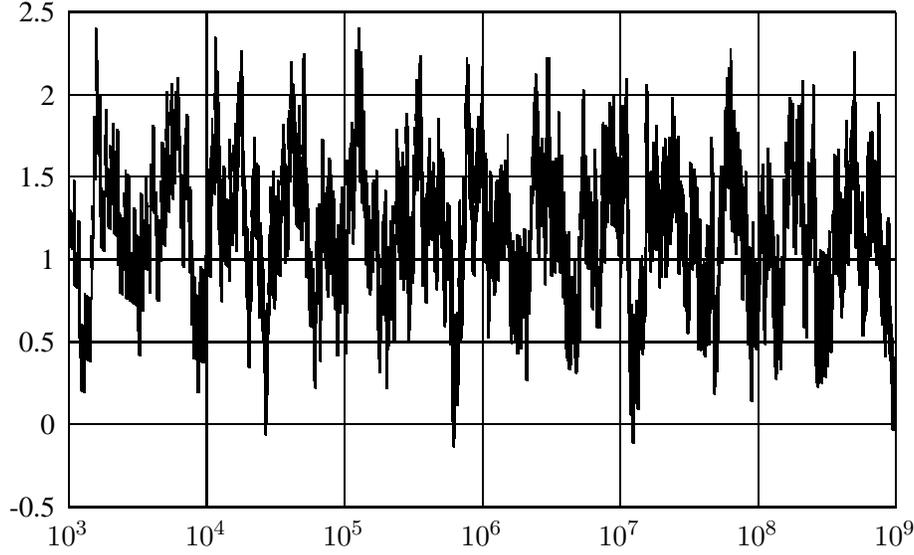}
\caption{ $\frac{\log x}{\sqrt{x}} \Delta(x;4,3,1)$ }
\end{figure}

\begin{figure}
\input{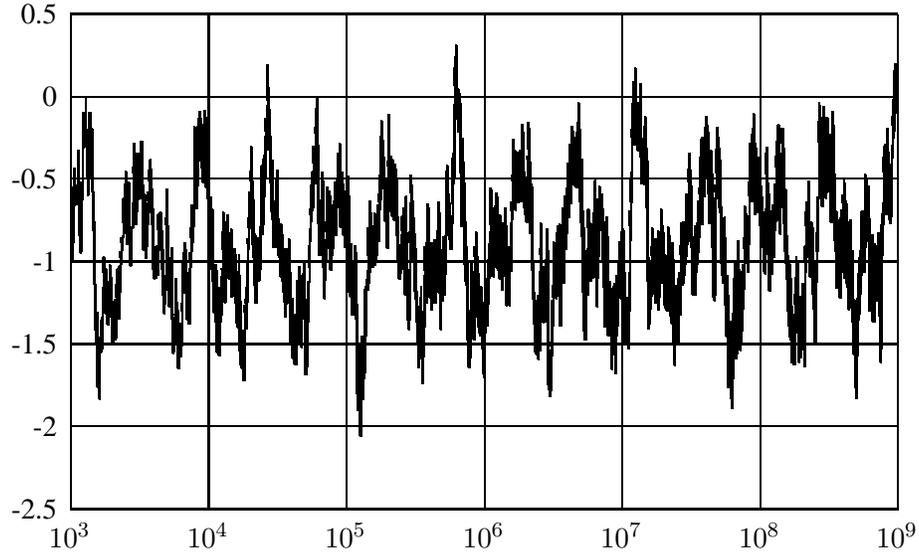}
\caption{ $\frac{\log x}{\sqrt{x}\log\log x} \Delta_2(x;4,3,1)$ }
\end{figure}

\begin{figure}
\input{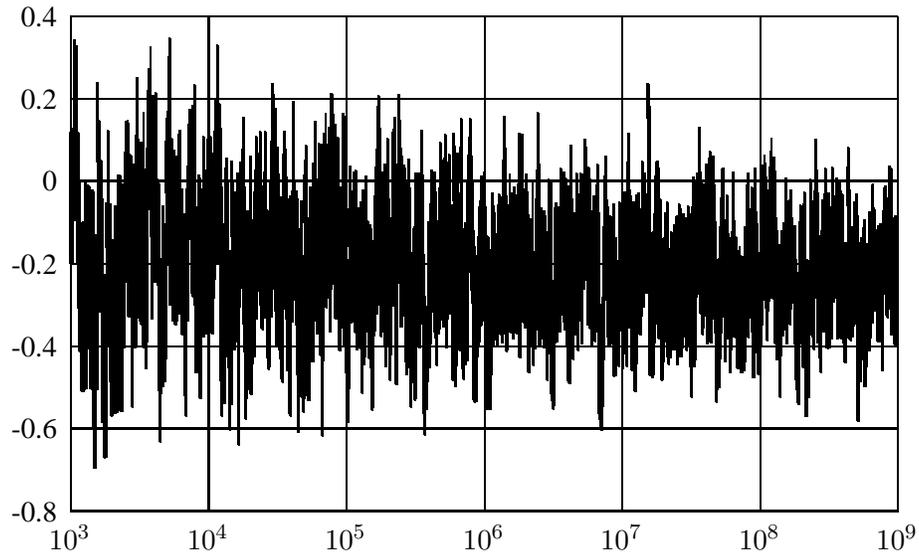}
\caption{$\Sigma(x;4,3,1)$}
\end{figure}

It is not immediate that Theorem \ref{thm1} follows from Theorem
\ref{DeltaDelta2}.  One first needs more precise information about 
the distribution of $\Delta(x;q,a,b)$ from \cite{RS}.  

\newtheorem*{theoremRS}{Theorem RS}\label{Theorem RS}
\begin{theoremRS}{\cite[\S 1]{RS}}\label{RSthm}
Assume ERH$_q$ and  GSH$_q$.   For any distinct $a,b\in A_q$, the
function
\be\label{Deltanorm}
\frac{u \Delta(e^u;q,a,b)}{e^{u/2}}
\ee
has a probabilistic distribution.  This distribution
(i) has mean $(N(q,b)-N(q,a))/\phi(q)$, (ii) is
symmetric with respect to its mean, and (iii) has
 a continuous density function. 
\end{theoremRS}

Assume $a$ is a quadratic nonresidue modulo $q$ and $b$ is a quadratic
residue.  Then $N(q,b)-N(q,a)>0$.
Let $f$ be the density function for the distribution of
\eqref{Deltanorm}, that is,
$$
f(t) = \frac{d}{dt} \lim_{U\to\infty} \frac{1}{U} \text{meas} \{ 0\le u\le U :
ue^{-u/2} \Delta(e^u;q,a,b) \le t \}.
$$
We see from Theorem RS that
$$
\delta(q,a,b) = \int_0^\infty f(t) \, dt > \frac12
$$
and from Theorem \ref{DeltaDelta2} that
$$
\delta_2(q,a,b) = \int_{-\infty}^{\frac{N(q,b)-N(q,a)}{2\phi(q)}} f(t) \, dt,
$$
from which Theorem \ref{thm1} follows.

Theorem \ref{DeltaDelta2} also determines the joint distribution
of any vector function
\be\label{vector2}
\frac{u}{e^{u/2}\log u} \( \Delta_2(e^u;q,a_1,b_1), \ldots,
\Delta_2(e^u;q,a_r,b_r) \).
\ee

\begin{thm}\label{thmjoint}
If $f(x_1,\ldots,x_r)$ is the density function of
$$
\frac{u}{e^{u/2}} \( \Delta(e^u;q,a_1,b_1), \ldots,
\Delta(e^u;q,a_r,b_r) \),
$$
then the joint density function of \eqref{vector2} is
$$
f\( \frac{N(q,b_1)-N(q,a_1)}{2\phi(q)} - x_1, \ldots,
\frac{N(q,b_r)-N(q,a_r)}{2\phi(q)} - x_r \). 
$$
\end{thm}

%%%%%%%%%%%%%%%%%%%%%%%%%%%%%%%%%%%%%%%%%
%
\subsection{Origin of Chebyshev's bias}
%
%%%%%%%%%%%%%%%%%%%%%%%%%%%%%%%%%%%%%%%%%

From an analytic point of view ($L$-functions), the
weighted sum
\be\label{DeltaLambda}
\Delta^*(x;q,a,b)=\sum_{\substack{n\le x \\ n\equiv a \bmod{q} }}
\Lambda(n) - \sum_{\substack{n\le x \\ n\equiv b \bmod{q} }}
\Lambda(n),
\ee
where $\Lambda$ is the von Mangoldt function, is more natural than 
\eqref{Delta}.  Expressing $\Delta^*(x;q,a,b)$ in terms of sums over
zeros of $L$-functions in the standard way (\S 19 of \cite{Da}), we
obtain, on ERH$_q$,
$$
e^{-u/2} \phi(q) \Delta^*(e^u;q,a,b) = - \sum_{\chi\in C(q)} \(
\overline{\chi}(a)-\overline{\chi}(b) \) \sum_{\gamma}
\frac{e^{i\gamma u}}{1/2+i\gamma} + O(u^2 e^{-u/2}),
$$
where  $\g$ runs over imaginary parts of nontrivial zeros of
$L(s,\chi)$ (counted with multiplicity). 
 Hypothesis GSH$_q$ implies, in particular, that $L(1/2,\chi)\ne 0$.
Each summand $e^{i\gamma u}/(1/2+i\gamma)$ is thus a harmonic
with mean zero as $u\to\infty$, and GSH$_q$ implies that
the harmonics behave independently.  Hence, we expect that $e^{-u/2}
\phi(q) \Delta^*(e^u;q,a,b)$ will behave like a mean zero random variable.
On the other hand, the right side of \eqref{DeltaLambda} contains not
only terms corresponding to prime $n$
but terms corresponding to powers of primes. 
Applying the prime number theorem for arithmetic progressions \eqref{PNT} to the
terms $n=p^2$ in \eqref{DeltaLambda} gives
$$
\Delta^*(x;q,a,b) = \sum_{\substack{p\le x \\ p\equiv a\bmod{q} }}
\log p -  \sum_{\substack{p\le x \\ p\equiv b\bmod{q} }}
\log p + \frac{x^{1/2}}{\phi(q)} \( N(q,a)-N(q,b) \) + O(x^{1/3}).
$$
Hence, on ERH$_q$ and GSH$_q$, we expect the expression
\be\label{sump}
\frac{1}{\sqrt{x}} \Biggl( \sum_{\substack{p\le x \\ p\equiv a\bmod{q} }}
\log p -  \sum_{\substack{p\le x \\ p\equiv b\bmod{q} }}
\log p \Biggr)
\ee
to behave like a random variable with mean $(N(q,b)-N(q,a))/\phi(q)$.  Finally,
the distribution of $\Delta(x;q,a,b)$ is obtained from the distribution
of \eqref{sump} and
partial summation.

%%%%%%%%%%%%%%%%%%%%%%%%%%%%%%%%%%%%%%%%%%%
%
\subsection{Analyzing $\Delta_2(x;q,a,b)$}
%
%%%%%%%%%%%%%%%%%%%%%%%%%%%%%%%%%%%%%%%%%%%%

A natural analog of $\Delta^*(x;q,a,b)$ is
\be\label{DeltaLambda2}
\sum_{\substack{mn\le x \\ mn\equiv a\bmod{q} }}
\Lambda(m)\Lambda(n)
- \sum_{\substack{mn\le x \\ mn\equiv b\bmod{q} }}
\Lambda(m)\Lambda(n).
\ee
As with $\Delta^*(x;q,a,b)$,
the expression in \eqref{DeltaLambda2} can be easily written as a
sum over zeros of $L$-functions plus a small error.
The main problem now is that the principal summands, namely $\log p_1
\log p_2$ for primes $p_1,p_2$,
are very irregular as a function of $p_1 p_2$, and thus
estimates for $\Delta_2(x;q,a,b)$ cannot be recovered by partial summation.
We get around this problem using a double integration, a method which
goes back to Landau \cite[\S 88]{La}.   We have
\be\label{DeltaG}
\begin{split}
\Delta_2(x;q,a,b) &= \frac{1}{\phi(q)} \sum_{\chi \in C(q)} \(
\overline{\chi}(a) -  \overline{\chi}(b) \) \sum_{\substack{n=p_1
    p_2\le x \\ p_1 \le p_2}} \chi(n) \\
&= \frac{1}{2\phi(q)}  \sum_{\chi \in C(q)} \(
\overline{\chi}(a) -  \overline{\chi}(b) \)
\int_0^\infty \int_0^\infty G(x,u,v;\chi) \, du\, dv +
O\pfrac{\sqrt{x}}{\log x},
\end{split}
\ee
where
\begin{equation} \label{Gsum}
G(x,u,v;\chi) = \sum_{p_1 p_2 \leq x} \frac{\chi(p_1 p_2) \log p_1
  \log p_2}{p_1^u p_2^v}.
\end{equation}
The related functions
$$
G^*(x,u,v;\chi) = \sum_{mn \leq x} \frac{\chi(mn) \Lambda(m)
  \Lambda(n)}{m^u n^v}
$$
are more ``natural'' from an analytic point of view, being easily
expressed in terms of zeros of Dirichlet $L$-functions.  By the
reasoning of the previous subsection, each $G^*(x,u,v;\chi)$ is
expected to be unbiased, the bias in $\Delta_2(x;q,a,b)$ originating
from the summands in  $G^*(x,u,v;\chi)$ where $m$ is not prime or $n$
is not prime.

%%%%%%%%%%%%%%%%%%%%%%%%%%%%%%%%%%%%%%%%%%%%%%%%%
%
\subsection{A heuristic argument for the bias in $\Delta_2(x;q,a,b)$}
%
%%%%%%%%%%%%%%%%%%%%%%%%%%%%%%%%%%%%%%%%%%%%%%%%%

We conclude this introduction with a heuristic
evaluation of the bias in $\Delta_2(x;q,a,b)$, which originates from
the difference between functions $G(x;u,v;\chi)$ and $G^*(x,u,v;\chi)$.
For simplicity of exposition, we'll concentrate on the special case
$(q,a,b)=(4,3,1)$.  In this case, the
bias arises from terms $p_1 p_2^2$ and $p_1^2 p_2^2$
which appear in  $G^*(x;u,v;\chi)$ but not in $G(x,u,v;\chi)$.
Let $\chi$ be the non-principal character modulo 4, so that
$$
\frac12 \int_0^\infty \int_0^\infty
 \bigl( G^*(x,u,v;\chi) - G(x,u,v;\chi) \bigr)\, du\, dv =
\frac12 \sum_{\substack{p_1^a p_2^b \le x \\ \max(a,b)\ge 2}}
\frac{\chi(p_1^a p_2^b)}{ab}.
$$
There are $O(x^{1/2}/\log x)$ terms with $\min(a,b)\ge 2$ and
$\max(a,b)\ge 3$.  By the prime
number theorem and partial summation,
$$
\frac12 \sum_{p_1^2 p_2^2 \le x} \frac14 = \frac18 \sum_{p\le
  \sqrt{x}} \pi \( \sqrt{x/p^2} \) 
\sim \frac{x^{1/2} \log\log x}{2\log x}.
$$
Thus,
\begin{align*}
\Delta_2(x;4,3,1) &= - \frac{1}{2} \sum_{mn \leq x} \frac{\chi(mn) \Lambda(m)
  \Lambda(n)}{\log m \log n} - \Bigg( \sum_{k=2}^{\infty} \frac{1}{k}
\sum_{p_1^k \leq x} \chi(p_1^k) \Delta(x/p_1^k;4,3,1) \Bigg) \\
&\qquad +\( \frac12+o(1) \)\frac{x^{1/2} \log\log x}{\log x}.
\end{align*}
By Theorem RS, $\Delta(y;4,3,1)=y^{1/2}/\log y + E(y)$, 
where $E(y)$ oscillates with mean 0.  Thus,
$$
\sum_{k=2}^{\infty} \frac{1}{k}
\sum_{p_1^k \leq x} \chi(p_1^k) \Delta(x/p_1^k;4,3,1) =
\sum_{k=2}^{\infty} \frac{2}{k}
\sum_{p_1^k \leq x} \chi(p_1^k) \frac{\sqrt{x/p_1^k}}{\log(x/p_1^k)} +
E'(x),
$$
where $E'(x)$ is expected to oscillate with mean zero.  The $k=2$ terms are
$$
\sum_{p_1^2 \leq x} \frac{\sqrt{x/p_1^2}}{\log(x/p_1^2)} \sim
\frac{\sqrt{x} \log\log x}{\log x},
$$
while the terms corresponding to $k\ge 3$ contribute 
$$
\ll \sum_{k=3}^{\infty} \frac{1}{k}
\sum_{p_1^k \leq x} \frac{\sqrt{x/p_1^k}}{\log(x/p_1^k)} \ll
\frac{\sqrt{x}}{\log x}.
$$
Thus, we find that
\begin{align*}
\Delta_2(x;4,3,1) &= - \frac{1}{2} \sum_{mn \leq x} \frac{\chi(mn) \Lambda(m)
  \Lambda(n)}{\log m \log n} -
\( \frac12+o(1) \)\frac{x^{1/2} \log\log x}{\log x} + E'(x).
\end{align*}

%%%%%%%%%%%%%%%%%%%%%%%%%%%%%%%
%
\subsection{Further problems}
%
%%%%%%%
It is natural to consider the distribution, in arithmetic
progressions, of numbers
composed of exactly $k$ prime factors, where $k\ge 3$ is fixed.  As with the
cases $k=1$ and  $k=2$, we expect there to be no bias if we count all
numbers $p_1^{a_1} p_2^{a_2} \cdots p_k^{a_k}$ with weight $(a_1
\cdots a_k)^{-1}$.  If, however, we count terms which are the product
of precisely $k$ primes (that is, numbers $p_1^{a_1} \cdots p_j^{a_j}$
with $a_1+\cdots+a_j=k$), then there will be a bias.  Hudson has
conjectured that the bias will be in the same direction as for primes
when $k$ is odd, and in the opposite direction for even $k$.  We
conjecture that, in addition, the bias becomes less pronounced as $k$
increases.

%%%%%%%%%%%%%%%%%%%%%%%%%%%%%%%%%%%%%
%
%
%
\section{Preliminaries}
%
%
%
%%%%%%%%%%%%%%%%%%%%%%%%%%%%%%%%%%%%

With $\chi$ fixed, the letter $\gamma$, with or without subscripts, denotes the
imaginary part of a zero of $L(s,\chi)$ inside the critical strip.  In
sums over $\g$, each term appears with its multiplicity $m(\g)$ unless we
specify that we sum over distinct $\g$.  Constants implied by $O-$ and
$\ll-$symbols depend only on $\chi$ (and hence, on $q$)
unless additional dependence is indicated with a subscript.  
Let 
$$
A(\chi) = \begin{cases} 1 & \chi^2=\chi_0 \\ 0 & \text{else} \end{cases},
$$
where $\chi_0$ is the principal character modulo $q$.  That is,
$A(\chi)=1$ if and only if $\chi$ is a real character.  
For $\chi\in C(q)$, define
$$
F(s,\chi)= \sum_{p} \frac{\chi(p) \log p}{p^s}.
$$
The following estimates are standard; see e.g. \cite[\S 15,16]{Da}.

\begin{lem} \label{FsHs} Let $\chi\in C(q)$,
assume ERH$_q$ and fix $c>\frac13$.
Then $F(s,\chi)= -\frac{L'}{L}(s,\chi) +
A(\chi) \frac{\zeta'}{\zeta}(2s) + H(s,\chi)$, where $H(s,\chi)$ is
analytic and uniformly  bounded in the half-plane $\Re s \ge c$.
\end{lem}

\begin{lem} \label{NTchi} 
Let $\chi$ be a Dirichlet character modulo $q$.
Let $N(T,\chi)$ denote the
number of zeros of $L(s, \chi)$ with $0<\Re s<1$ and
$|\Im s|<T$.   Then
\begin{enumerate}
\item  $N(T,\chi)=O(T\log(qT))$ for $T\ge 1$.
\item  $N(T,\chi) - N(T-1,\chi) = O(\log(qT))$ for $T\ge 1$.
\item  Uniformly for $s=\sigma+it$ and $\sigma\ge -1$,
$$
\frac{L'(s,\chi)}{L(s,\chi)} =
  \sum_{|\gamma-t|<1} \frac{1}{s-\rho} + O(\log q(|t|+2)).
$$
\item $-\frac{\zeta'}{\zeta}(\sigma)  = \frac{1}{\sigma -1}+O(1) $ uniformly
  for $\sigma\ge \frac12$, $\sigma\ne 1$.
\item $\big|\frac{\zeta'}{\zeta}(\sigma+iT) \big| \le
  -\frac{\zeta'}{\zeta}(\sigma)$ for $\sigma > 1$.
\end{enumerate}
\end{lem}

For a suitably small, fixed $\delta>0$, we say that a number $T\ge 2$ is
\emph{admissible} if for all $\chi\in C(q)\cup \{\chi_0\}$ and all zeros $\frac12+i\g$ of $L(s,\chi)$,
$|\g-T| \ge \delta (\log T)^{-1}$.
By Lemma \ref{NTchi}, we can choose $\del$ small enough, depending
on $q$, so that
there is an admissible $T$ in $[U,U+1]$ for all $U\ge 2$.
From Lemma \ref{NTchi} we obtain

\begin{lem}\label{zeta} 
 Uniformly for $\sigma \ge \frac25$ and admissible $T\ge 2$,
$$
|F(\sigma+iT,\chi)| = O(\log^2 T).
$$
\end{lem}

\begin{lem} \label{sumg1g2}
Fix $\chi\in C(q)$ and assume $L(\frac12,\chi)\ne 0$.
For $A\ge 0$ and real $k\ge 0$,
$$
\sum_{\substack{ |\g_1|, |\g_2| \ge A \\ |\g_1-\g_2| \geq
     1}} \frac{{\log}^k (|\g_1|+3) {\log}^k (|\g_2|+3)}{{|\g_1|}
   {|\g_2|} |\g_1-\g_2|} \ll_{k} \frac{(\log (A+3))^{2k+3}}{A+1}.
$$
\end{lem}

\begin{proof}
The sum in question is at most twice the sum 
of terms with $|\g_2| \ge |\g_1|$, which is
$$
\ll \sum_{|\g_2|\ge A}
\frac{{\log}^{2k} (|\g_2|+3)}{{|\g_2|}} \bigg(\frac{1}{|\g_2|}
\sum_{|\g_1| < \frac{|\g_2|}{2}} \frac{1}{{|\g_1|}} +
\frac{1}{{|\g_2|}} \sum_{\substack{\frac{|\g_2|}{2} \leq |\g_1| \leq
  |\g_2| \\ |\g_2-\g_1| \ge 1}} \frac{1}{|\g_2-\g_1|} \bigg).
$$
By Lemma \ref{NTchi} (1), the two sums over $\g_1$ are $O(\log^2 (|\g_2|+3))$.  A further
application of Lemma \ref{NTchi} (1) completes the proof.
\end{proof}

We conclude this section with
a truncated version of the Perron formula for $G(x,u,v;\chi)$.

\begin{lem}\label{Perron} Uniformly for $x \leq T \leq 2x^2$,
  $x \ge 2$, $u\ge 0$ and $v\ge 0$, we have
\begin{equation} \label{Ginterror}
G(x,u,v;\chi) = \frac{1}{2\pi i} \int_{c-iT}^{c+iT} F(s+u,\chi)F(s+v,\chi)
\frac{x^s}{s} ds + O(\log^3 x),
\end{equation}
where $c=1+\frac{1}{\log x}$.
\end{lem}

\begin{proof}
For $\Re s > 1$, we have
$$
F(s+u,\chi)F(s+v,\chi) = \sum_{n=1}^\infty f(n) n^{-s}, \qquad
f(n)=\sum_{p_1p_2=n} \frac{\chi(p_1 p_2) \log p_1 \log p_2}{p_1^{u} p_2^v}.
$$
Using the trivial estimate
$|f(n)| \le \log^2 n$
and a standard argument \cite[\S 17, (3) and (5)]{Da}, we
obtain the desired bounds.
\end{proof}

%%%%%%%%%%%%%%%%%%%%%%%%%%%%%%%%%%%%%%%%%%%%%%%%%%%%%%%%%%%%%
%
%
%
\section{Outline of the proof of Theorem \ref{DeltaDelta2}}\label{sec:outline}
%
%
%
%%%%%%%%%%%%%%%%%%%%%%%%%%%%%%%%%%%%%%%%%%%%%%%%%%%%%%%%%%%%%

Throughout the remainder of this paper, fix $q$, assume ERH$_q$
and that $L(\frac12,\chi) \ne 0$ for each $\chi\in C(q)$.  Let
$$
\eps = \frac{1}{100}.
$$
We next define a function $T(x)$ as follows.  For each positive
integer $n$, let $T_n$ be an admissible value of $T$ 
satisfying $\exp (2^{n+1}) \le T_n \le
\exp(2^{n+1})+1$ and set $T(x)=T_n$ for $\exp(2^n) < x \le \exp
(2^{n+1})$.  In particular, we have 
$$
x \le T(x) \le 2x^2 \qquad (x\ge e^2).
$$

Our first task is to express the double integrals in \eqref{DeltaG}
in terms of sums over zeros of $L(s,\chi)$.  This is proved in
Section \ref{sec:analytic}.

\begin{lem}\label{analytic}
Let $\chi\in C(q)$ and let $T=T(x)$.  Then 
\begin{multline*}
 x^{-1/2} \int_0^\infty \int_0^\infty G(x,u,v;\chi)\, du\, dv \\
= 2\int_{0}^{2\eps} \!\!\! \int_{0}^{2\eps} \sum_{|\g|\le T} 
\frac{F(\frac{1}{2}+u-v+i\g,\chi)x^{-v+i\g}}{\frac{1}{2}-v+i\g}
du\, dv + \frac{A(\chi) \log\log x + \Sigma_1(x;\chi)+O(1)}{\log x},
\end{multline*}
where
$\int_{1}^{Y} |\Sigma_1 (e^y;\chi)|^2 dy = O(Y)$.
\end{lem}

The aggregate of terms $A(\chi)\log\log x/\log x$ account for the 
bias for products of two primes.  As with the Chebyshev bias for primes,
these terms arise from poles of $F(s)$ at $s=\frac12$ when $A(\chi)=1$
(see Lemma \ref{FsHs}) and correspond to the contribution to $F(s)$
from squares of primes.
The double integral on the right side in Lemma \ref{analytic} is
complicated to analyze.  In Section \ref{sec:doubleint} we prove the following.

\begin{lem}\label{doubleint}  Let $\chi\in C(q)$.
Let $n$ be a positive integer, $2^n < \log x \le 2^{n+1}$ and  $T=T(x)$.  Then
\begin{align*}
2 \int_0^{2 \eps} \int_0^{2 \eps} &\sum_{|\g| \leq T} 
  \frac{ F(\frac{1}{2}+u-v+i \g,\chi) x^{-v+i \g}}{\frac{1}{2}-v+i \g} 
  du\ dv 
= \frac{\Sigma_2(x;\chi)}{\log x}\\
& + 2\sum_{\substack{|\g| \leq T \\
\g \text{ distinct}}} m^2(\g) x^{i\g} (\frac{1}{2}+i\g) 
\int_0^{2 \eps-2^{-n}} \frac{ x^{-v}}{\frac{1}{2}-v+i \g} 
\int_{v+2^{-n}}^{2 \eps} \frac{du}{(u-v)(\frac{1}{2}-u+i \g)} dv,
\end{align*}
where $\int_{1}^{Y} |\Sigma_2(e^y;\chi)|^2 dy = o(Y\log^2 Y)$.
\end{lem}

The terms on the right in Lemma \ref{doubleint} with small $|\g|$ 
will give the main term, and terms with larger $|\g|$ are considered 
as error terms.
The next lemma is proved in Section \ref{sec:TT0}.

\begin{lem}\label{TT0}
Let $\chi\in C(q)$.
Let $n$  be a positive integer, $2^n < \log x \le 2^{n+1}$,  $T=T(x)$
and $2\le T_0 \le T$.  Then
\begin{align*}
&2\sum_{\substack{|\g| \leq T \\ \g \text{ distinct}}} m^2(\g) x^{i\g} 
(\frac{1}{2}+i\g) \int_0^{2 \eps-2^{-n}} \frac{ x^{-v}}{\frac{1}{2}-v+i \g}
 \int_{v+2^{-n}}^{2 \eps} \frac{du}{(u-v)(\frac{1}{2}-u+i \g)} dv \\
&\qquad = \frac{2\log\log x}{\log x} \sum_{\substack{|\gamma|\le T_0 \\ 
\g \text{ distinct}}}
 \frac{m^2(\gamma) x^{i\gamma}}{1/2+i\g} + O\pfrac{\log^3 T_0}{\log x} + 
\frac{\Sigma_3(x,T_0;\chi)}{\log x},
\end{align*}
where 
$$
\frac{1}{Y}\int_1^Y |\Sigma_3(e^y,T_0;\chi)|^2\, dy 
\ll \frac{\log^5 T_0}{T_0} \log^2 Y.
$$
\end{lem}

Combining Lemmas \ref{analytic}, \ref{doubleint} and \ref{TT0}
with \eqref{DeltaG} yields (for fixed, large $T_0$)
\begin{multline*}
\Delta_2(x;q,a,b) = \frac{\sqrt{x}}{2\phi(q)} \sum_{\chi\in C(q)}
\( \overline{\chi}(a)-\overline{\chi}(b) \) \Bigg[
\frac{\log\log x}{\log x} \Bigg( A(\chi)+2 \sum_{\substack{|\g|\le T_0
\\ \g \text{ distinct}}} \frac{m^2(\g) x^{i\g}}{1/2+i\g}\Bigg) \\
+ \frac{\Sigma_1(x;\chi)+\Sigma_2(x;\chi)+\Sigma_3(x,T_0;\chi)
+O(\log^3 T_0)}{\log x}
\Bigg],
\end{multline*}
where
$$
\lim_{T_0\to \infty} \(\limsup_{Y\to \infty} 
\frac{1}{Y\log^2 Y} \sum_{\chi\in C(q)}  \int_{1}^Y |\Sigma_1(e^y;\chi)+
\Sigma_2(e^y;\chi)+\Sigma_3(e^y;T_0;\chi)|^2\, dy
\) =0.
$$
On the other hand (cf. \cite{RS}),
$$
\Delta(x;q,a,b) = \frac{\sqrt{x}}{\log x} \( \frac{N(q,b)-N(q,a)}{\phi(q)}
- \sum_{\chi\in C(q)} \( \overline{\chi}(a)-\overline{\chi}(b) \) 
\sum_{|\g| \le T_0}\frac{x^{i\g}}{1/2+i\g} + \Sigma_4(x;T_0) \),
$$
where 
$$
\lim_{T_0\to \infty} \(\limsup_{Y\to\infty} Y^{-1} \int_{1}^Y |\Sigma_4(e^y;T_0)|^2\, dy
\) =0.
$$
Now assume $m(\g)=1$ for all $\g$, and note that
$$
 \sum_{\chi\in C(q)}
\( \overline{\chi}(a)-\overline{\chi}(b) \) 
 A(\chi) = N(q,a)-N(q,b).
$$
Letting $T_0\to \infty$ finishes the proof of Theorem \ref{DeltaDelta2}.

%%%%%%%%%%%%%%%%%%%%%%%%%%%%%%%%%%%%%%%%%%%%%%%%%%%%
%
%
\section{proof of Lemma \ref{analytic}}\label{sec:analytic}
%
%
%%%%%%%%%%%%%%%%%%%%%%%%%%%%%%%%%%%%%%%%%%%%%%%%%%%%%

Assume ERH$_q$ throughout.
We first estimate $G(x,u,v;\chi)$ for different ranges of $u,v$.

\begin{lem}\label{Gbounds}
Let $\chi\in C(q)$, $\chi\ne \chi_0$.
For $x\ge 4$, the following hold:
\begin{enumerate}

\item For $u,v \geq \eps$,
$G(x,u,v;\chi) \ll x^{\frac{1}{2} - \frac {\eps}{2}} \log^5 x$.

\item For $u \geq 2\eps$, $v \leq \eps$ and $T=T(x)$,
\begin{align*}
x^{-1/2} G(x,u,v;\chi) &= \sum_{|\g|\le T} \frac{F(\frac{1}{2}+u-v+i\g
  ,\chi)x^{-v+i\g}}{\frac{1}{2}-v+i\g}
-A(\chi) \frac{F(\frac{1}{2}+u-v,\chi)x^{-v}}{1-2v} \\
&\qquad\qquad + O(x^{-\frac{3\eps}{2}} \log^5 x).
\end{align*}

\item For $u\le 2\eps$, $v\le 2\eps$, $u\ne v$
and $T=T(x)$,
\begin{multline*}
x^{-1/2} G(x,u,v;\chi) = \sum_{|\g| \le T}
 \frac{F(\frac{1}{2}+u-v+i\g
   ,\chi)x^{-v+i\g}}{\frac{1}{2}-v+i\g} +
 \frac{F(\frac{1}{2}-u+v+i\g
   ,\chi)x^{-u+i\g}}{\frac{1}{2}-u+i\g} \\  \qquad -
  A(\chi) \( \frac{F(\frac{1}{2}+u-v,\chi)x^{-v}}{1-2v}+
  \frac{F(\frac{1}{2}-u+v,\chi)x^{-u}}{1-2u} \)
  + O (x^{-3\eps} \log^5 x).
\end{multline*}
\end{enumerate}
\end{lem}

\begin{proof}
Assume $u\ge \eps$ and $v\ge \eps$.  Start with the approximation of
$G(x,u,v;\chi)$ given by Lemma \ref{Perron}, then deform the segment
of integration to the contour consisting of three straight segments connecting
$c-iT$, $b-iT$, $b+iT$ and $c+iT$, where $b=\frac12 - \frac{\eps}{2}$
and $T=T(x)$.  The rectangle formed by the new and old contours
does not contain any poles of $F(s+u,\chi)F(s+v,\chi)s^{-1}$.
On the three new segments, by Lemmas \ref{FsHs}, \ref{NTchi} and \ref{zeta}, we
have $|F(s+u,\chi)F(s+v,\chi)| \ll \log^4 T$.  Hence the integral of
$F(s+u,\chi)F(s+v,\chi) x^s s^{-1}$ over the three segments is
$$
\ll (\log^4 x) \Bigl( \int_b^c \frac{x^\sigma}{|\sigma+iT|}\, d\sigma
+ \int_{-T}^T \frac{x^b}{|b+it|}\, dt \Bigr) \ll x^b \log^5 x.
$$
This proves (1).

We now consider the case $v \leq \eps$ and $u \geq 2\eps$.  
We set $b = \frac{1}{2}-\frac{3\eps}{2}$ and deform the
contour of integration as in the previous case.  
Since $u+b \ge \frac12 + \frac{\eps}{2}$ and $v+b \le \frac12 -
\frac{\eps}{2}$, we have by Lemma \ref{zeta} that
$|F(s+u,\chi)F(s+v,\chi)| \ll \log^4 T \ll \log^4 x$
on all three new segments.  As in the proof of (1), the integral over
the new contour is $\ll x^b \log^5 x$.
We pick up residue terms from poles of
$F(s+v,\chi)$ inside the rectangle coming from the nontrivial zeros of
$L(s,\chi)$, plus a pole at $s=\frac12-v$ from
the $\frac{\zeta'}{\zeta}(2s+2v)$ term if $\chi^2=\chi_0$.  The sum of
the residues is
$$
\sum_{|\g|\le T}
\frac{F(\frac{1}{2}+u-v+i\g
  ,\chi)x^{\frac{1}{2}-v+i\g}}{\frac{1}{2}-v+i\g}
-A(\chi) \frac{F(\frac{1}{2}+u-v,\chi)x^{\frac{1}{2}-v}}{1-2v},
$$
and (2) follows.

Finally, consider the case $0\le u,v \le 2\eps$. 
Let $b = \frac{1}{2} - 3\eps$ and deform the contour as in the
previous cases.  As before, the integral over the new contour is
$O(x^b \log^5 x)$.  This time, we pick up 
residues from poles of both $F(s+u,\chi)$ and $F(s+v,\chi)$.  The sum of the
residues is
\begin{align*}
\sum_{|\g|\le T}  \Bigl( \frac{F(\frac{1}{2}+u-v+i\g,\chi) 
x^{\frac{1}{2}-v+i\g}}{\frac{1}{2}-v+i\g}+ \frac{F(\frac{1}{2}-u+v+i\g
   ,\chi)x^{\frac{1}{2}-u+i\g}}{\frac{1}{2}-u+i\g} \Bigr) \\
- A(\chi) \( \frac{F(\frac{1}{2}+u-v,\chi)x^{\frac{1}{2}-v}}{1-2v}+
\frac{F(\frac{1}{2}-u+v,\chi)x^{\frac{1}{2}-u}}{1-2u} \),
\end{align*}
and (3) follows.
\end{proof}  %%% of G bounds 

\begin{proof}[Proof of Lemma \ref{analytic}]
Begin with
$$
 \int_0^\infty \int_0^\infty G(x,u,v;\chi)\, du\, dv  = I_1 + I_2 +
 2I_3 + I_4,
$$
where $I_1$ is the integral over  $\max(u,v)\ge \log x$, 
$I_2$ is the integral over $2\eps \le \max(u,v) \le \log
x$ and $\min(u,v)\ge \eps$, $I_3$ is the integral over $0\le v\le
\eps$, $2\eps \le u \le \log x$, and $I_4$ is the integral over $0\le
u,v \le 2\eps$.  For $\max(u,v)\ge \log x$, 
$$
|G(x,u,v;\chi)| \le \sum_{p\le x} \frac{\log p}{p^u} \sum_{p\le x}
\frac{\log q}{q^v} \ll \frac{x}{2^{\max(u,v)}},
$$
whence $I_1 \ll x^{1-\log 2}$.  By Lemma \ref{Gbounds} (1), $I_2 \ll
x^{1/2-\eps/2} \log^7 x$.  

By Lemma \ref{Gbounds} (2), 
\be\label{I3}
\begin{split}
I_3 &= x^{1/2} \int_0^\eps \int_{2\eps}^{\log x} 
\sum_{|\g|\le T} \frac{F(\frac{1}{2}+u-v+i\g
  ,\chi)x^{-v+i\g}}{\frac{1}{2}-v+i\g}
-A(\chi) \frac{F(\frac{1}{2}+u-v,\chi)x^{-v}}{1-2v}\, du\, dv \\
&\qquad\qquad + O(x^{1/2-\frac{3\eps}{2}} \log^6 x).
\end{split}
\ee
By Lemmas \ref{NTchi} and \ref{zeta},
\be\label{I31}
\int_{0}^{\eps} \int_{2\eps}^{\log x} \frac{F(\frac{1}{2}+u-v
  ,\chi)x^{-v}}{1-2v} du\ dv \ll \int_{0}^{\eps} x^{-v}\, dv
\ll \frac{1}{\log x}.
\ee
Let
$$
\Sigma_1(x) = (\log x) \int_0^\eps \int_{2\eps}^{\log x} 
\sum_{0<|\g|<T} \frac{F(\frac{1}{2}+u-v+i\g
  ,\chi)x^{-v+i\g}}{\frac{1}{2}-v+i\g}\, du\, dv.
$$
Since $\frac{1}{2}+u-v \ge \frac{1}{2}+\eps$ for $0 \le v  \le \eps$ and $2
\eps \le u \le \log x$, by Lemmas \ref{FsHs}, \ref{NTchi}, and
\ref{zeta}, 
$$
F(\frac{1}{2}+u-v+i\g,\chi) = -\frac{L'}{L}(\frac{1}{2}+u-v+i\g,
\chi) + O(1) \ll \log(|\g|+3).
$$
We also have $F(1/2+u-v+i\g,\chi) \ll 2^{-u}$ for $u\ge 2$.
Thus, for positive integers $n$,
\begin{align*}
\int_{2^n}^{2^{n+1}} |\Sigma_1 (e^y)|^2 dy &\ll 2^{2n}
\sum_{|\g_1|,|\g_2|\le T} \frac{\log (|\g_1|+3)
  \log(|\g_2|+3)}{|\g_1 \g_2|} \\
& \qquad\qquad \times \int_{0}^{\eps}
\int_{0}^{\eps}  \Bigg| \int_{2^n}^{2^{n+1}}
e^{y(-v_1+i\g_1-v_2-i\g_2)} dy \Bigg| dv_1 dv_2.
\end{align*}
The summands with $|\g_1-\g_2| <  1$ contribute, by Lemma
\ref{NTchi},
\begin{align*}
&\ll 2^{2n} \sum_{\substack{|\g_1|, |\g_2| \le T
      \\ |\g_1-\g_2| < 1}}   \frac{\log (|\g_1|+3) \log
    (|\g_2|+3)}{|\g_1||\g_2|} \int_{2^n}^{2^{n+1}}
  \Big(\int_{0}^{\eps} e^{-vy} dv\Big)^2 dy \\
&\ll 2^n \sum_{|\g|\le  T }  \frac{{\log}^3
  (|\g|+3)}{|\g|^2} \ll 2^n.
\end{align*}
The summands with $|\g_1-\g_2| \ge 1$ contribute, by Lemma \ref{sumg1g2},
\[
\ll \sum_{\substack{|\g_1|, |\g_2| < T \\ |\g_1-\g_2|
    \geq 1}} \frac{2^{2n} \log (|\g_1|+3) \log
  (|\g_2|+3)}{|\g_1||\g_2||\g_1-\g_2|}
\Big(\int_{0}^{\eps} e^{-v2^n} dv \Big)^2
 \ll 1.
\]
Thus, $\int_{2^n}^{2^{n+1}} |\Sigma_{1} (e^y)|^2 dy = O (2^n)$.
Summing over $n\le \frac{\log Y}{\log 2}+1$ yields
$\int_1^Y |\Sigma_1(e^y)|^2\, dy =  O(Y)$.

Finally, using Lemma \ref{Gbounds} (3) gives
\be\label{I4}
\begin{split}
I_4 &= x^{1/2} \int_0^{2\eps} \int_0^{2\eps} 
 \sum_{|\g|\le T}
 \frac{F(\frac{1}{2}+u-v+i\g
   ,\chi)x^{-v+i\g}}{\frac{1}{2}-v+i\g} +
 \frac{F(\frac{1}{2}-u+v+i\g
   ,\chi)x^{-u+i\g}}{\frac{1}{2}-u+i\g} \\ 
& \qquad -
  A(\chi) \( \frac{F(\frac{1}{2}+u-v,\chi)x^{-v}}{1-2v}+
  \frac{F(\frac{1}{2}-u+v,\chi)x^{-u}}{1-2u} \) \, du\, dv
  + O (x^{\frac12-3\eps} \log^3 x).
\end{split}
\ee
Now assume $\chi^2=\chi_0$.  We will show that
\be\label{I41}
- \int_0^{2\eps} \int_0^{2\eps}
 \frac{F(\frac{1}{2}+u-v,\chi)x^{-v}}{1-2v}+ 
  \frac{F(\frac{1}{2}-u+v,\chi)x^{-u}}{1-2u} \, du\, dv
= \frac{\log\log x  + O(1)}{\log x}.
\ee
Together with \eqref{I3}, \eqref{I31} and \eqref{I4}, this completes
the proof of Lemma \ref{analytic}.

Note that $F(\frac{1}{2}+w)=-\frac{1}{2w}+O(1)$ by Lemmas \ref{FsHs}
and \ref{zeta}.
Replacing $x$ with $e^y$, the left side of \eqref{I41} is
$$
=\frac12 \int_{0}^{2\eps} \int_{0}^{2\eps}
  \frac{e^{-yv}}{(u-v)(1-2v)} +
  \frac{e^{-yu}}{(v-u)(1-2u)}du\ dv + O\Big(\int_{0}^{2\eps}
  \int_{0}^{2\eps} e^{-yv} du\ dv\Big).
$$
The error term above is $O(1/y)$.
In the main term, when $|u-v|<1/y$, the integrand is $O(y e^{-vy})$
and the corresponding part of the double integral is $O(1/y)$.
When $u \ge v + 1/y$, the integrand is
$$
\frac{e^{-vy}}{u-v} + O\pfrac{v e^{-vy}+e^{-uy}}{u-v}
$$
and the corresponding part of the double integral is
$$
\int_0^{2\eps} e^{-vy} \log\pfrac{y}{2\eps-v}\, dv + O\pfrac{1}{y} =
\frac{\log y + O(1)}{y}.
$$
The contribution from $u\le v-1/y$ is, by symmetry, also $\frac{\log y+O(1)}{y}$. 
The asymptotic \eqref{I41} follows.
\end{proof}

%%%%%%%%%%%%%%%%%%%%%%%%%%%%%%%%%%%%%%%%%%%%%%
%
%
%
%
\section{Proof of Lemma \ref{doubleint}}\label{sec:doubleint}
%
%
%
%
%%%%%%%%%%%%%%%%%%%%%%%%%%%%%%%%%%%%%%%%%%%%%%

\begin{lem}\label{Qint}
 Uniformly for $y\ge 1$, $0 < |\xi| \leq 1$, $|w| \geq \frac{1}{2}$
 and $a\ge 0$ we have
$$
\Biggl| \int_0^{2 \eps} \int_0^{2 \eps} \frac{v^a e^{-vy}}{(u-v+i\xi)(w-v)} 
 du\ dv \Biggr|   \ll \frac{(4\eps)^{a}\log \min(2y, \frac{2}{|\xi|})}{y|w|}.
$$
\end{lem}

\begin{proof}  Let $I$ denote the double integral in the
Lemma.  If $|\xi| \geq \frac{1}{y}$, then
\begin{align*}
I&\ll \frac{1}{|w|} \int_0^{2 \eps} v^a e^{-vy} \int_0^{2\eps} 
\min\( \frac{1}{|u-v|}, \frac{1}{|\xi|} \)\, du\, dv \\
&\ll \frac{(2\eps)^a}{|w|} \(1 + \log \frac{2}{|\xi|}\)
 \int_0^{2 \eps} e^{-vy}\, dv
\ll \frac{(2\eps)^a \log({\frac{2}{|\xi|}})}{y|w|}.
\end{align*}
If $|\xi| < \frac{1}{y}$, let $I=I_1+I_2+I_3$, where $I_1$ is the part of
 $I$ coming from $|u-v| \leq |\xi|$, $I_2$ is the part of $I$ coming from
 $|\xi| < |u-v| \leq \frac{1}{y}$, and $I_3$ is the part of $I$ coming
 from $|u-v|>\frac{1}{y}$.  We have
$$
I_1 \ll \frac{1}{|w\xi|} \; \;\;  \iint 
\limits_{\substack{0\le u,v\le 2\eps \\ |u-v|\leq |\xi|}}
v^a e^{-vy} du\ dv \ll \frac{(2\eps)^a}{y|w|}.
$$
and
$$
I_3 \ll \frac{(2\eps)^a}{|w|}\;\;  \iint 
\limits_{\substack{0\le u,v \le 2\eps \\ |u-v|\geq \frac{1}{y}}} 
\frac{e^{-vy}}{|u-v|} du\ dv
\ll \frac{(2\eps)^a}{|w|} \int_0^{2 \eps} e^{-vy}(\log y + 1) dv 
\ll \frac{(2\eps)^a \log(2y)}{y|w|}.
$$
By symmetry,
$$
I_2 = \frac12 \iint\limits_{|\xi| < |u-v| \le 1/y} \frac{v^a
  e^{-vy}}{(u-v+i\xi)(w-v)} + \frac{u^a e^{-uy}}{(v-u+i\xi)(w-u)}\,
  du\, dv. 
$$
Since, $|u^a-v^a|\le a|u-v| (2\eps)^{a-1}$,
\be\label{euy}\begin{split}
u^a e^{-uy} - v^a e^{-vy} &= e^{-vy} v^a \( e^{(v-u)y}-1 \) +
e^{-vy} (u^a-v^a) e^{(v-u)y} \\&\ll e^{-vy} y |u-v| (4\eps)^a.
\end{split}\ee
We deduce that
\begin{align*}
I_2&= \!\!\!\iint\limits_{\substack{0\le u,v\le 2\eps \\ |\xi| < |u-v| \le 1/y}} 
 \!\! \frac{(w-u)(u-v)(u^ae^{-uy}-v^ae^{-vy})+u^ae^{-uy}(u-v)^2+O(|\xi w| (2\eps)^a
  e^{-vy})}{2(u-v+i\xi)(v-u+i\xi)(w-u)(w-v)} du\ dv \\
&\ll \frac{(4\eps)^a}{|w|}\;\;\; 
 \iint\limits_{\substack{0\le u,v\le 2\eps \\ |\xi|<|u-v| \le 1/y}}
y e^{-vy} + \frac{|\xi| e^{-vy}}{|u-v|^2} \, du\, dv
\ll \frac{(4\eps)^a}{y|w|}.
\end{align*}
\end{proof}

%%%%%%%%%%%%%%%%%%%%%%%%%%%%%%%%%%%%%%%%%%%%%%%%%%%%%%%%%%%%%

\begin{proof}[Proof of Lemma \ref{doubleint}]
Let $y=\log x$.
We first note by Lemmas \ref{FsHs} and \ref{NTchi},
$$
F(\frac{1}{2}+u-v +i \g,\chi) = \frac{m(\g)}{u-v}+R(\g,u-v)+R'(\g,u-v),
$$
where
$$
R(\g,w) = \sum_{0< |\g'-\g| \leq 1} \frac{1}{w+i(\g - \g')}, \qquad 
R'(\g,u-v) =  O(\log (|\g|+3)).
$$
Then, the double integral in Lemma \ref{doubleint} is
$$
=\sum_{i=1}^4 \Sigma_{2,i}(y)+2 \sum_{\substack{|\g| \leq T \\
\g \text{ distinct}}} m^2(\g) e^{iy\g} 
(\frac{1}{2}+i\g) \int_0^{2 \eps-2^{-n}} \frac{ e^{-yv}}{\frac{1}{2}-v+i \g}
\int_{v+2^{-n}}^{2 \eps} \frac{du}{(u-v)(\frac{1}{2}-u+i \g)} dv,
$$
where
\begin{align*}
\Sigma_{2,1}(y) &= 2 \int_0^{2 \eps} \int_0^{2 \eps} \sum_{|\g| \leq T} 
\frac{ R(\g,u-v) e^{y(-v+i \g)}}{\frac{1}{2}-v+i \g} du\ dv, \\
\Sigma_{2,2}(y) &= 2  \int_0^{2 \eps} \int_0^{2 \eps}  
\frac{ R'(\g,u-v) e^{y(-v+i \g)}}{\frac{1}{2}-v+i \g}\, du\, dv, \\
\Sigma_{2,3}(y)& =  \sum_{\substack{|\g| \leq T \\
\g \text{ distinct}}} m^2(\g) e^{iy \g} (\frac{1}{2} + i\g)
 \iint \limits_{\substack{0 \leq u,v \leq 2 \eps \\ |u-v| \leq 2^{-n}}} 
\frac{ e^{-yv}-e^{-uy}}{(u-v)(\frac{1}{2}-v+i \g)(\frac{1}{2}-u+i \g)} dv\ du, \\
\Sigma_{2,4}(y) &= 2 \sum_{\substack{|\g| \leq T \\
\g \text{ distinct}}} m^2(\g) e^{iy \g} (\frac{1}{2} + i\g)
 \int_{2^{-n}}^{2 \eps} \int_{0}^{v-2^{-n}}  
\frac{ e^{-yv}}{(u-v)(\frac{1}{2}-v+i \g)(\frac{1}{2}-u+i \g)} du\ dv.
\end{align*}
We show that $\sum_{j=1}^4 \Sigma_{2,j}(y)$ is small in mean square.  
Note that for  $2^n < y \le 2^{n+1}$, $T=T(e^y)$ is constant.  
Also, by Lemma \ref{NTchi}, we have 
\be\label{mgam}
m(\g) \ll \log (|\g|+3).
\ee

%%%%%%%%%%%      Sigma_{2,2}

First, by Lemmas \ref{NTchi} and \ref{sumg1g2},
\be\label{Sigma22}
\begin{split}
\int_{2^n}^{2^{n+1}}  \!\! &|\Sigma_{2,2}(y)|^2\, dy = 4
\iiiint\limits_{[0,2\eps]^4} \!\!\! \sum_{\substack{|\g_1|\le T \\ |\g_2| \le T}}
\frac{R'(\g_1,u_1-v_1)\overline{R'(\g_2,u_2-v_2)}}
{(\frac12-v_1+i\g_1)(\frac12-v_2-i\g_2)} \\
&\qquad\qquad \times
\int_{2^n}^{2^{n+1}} \!\!\!\! e^{y(-v_1-v_2+i\g_1-i\g_2)}\, dy \, du_j dv_j \\
&\ll \sum_{|\g_1-\g_2|>1} \frac{\log (|\g_1|+3) \log (|\g_2|+3)}
{|\g_1\g_2|\cdot|\g_1-\g_2|} \iiiint\limits_{[0,2\eps]^4}
 e^{-2^n(v_1+v_2)}\, du_j dv_j \\
&\qquad +  \sum_{|\g_1-\g_2|\le 1} \frac{\log (|\g_1|+3) \log (|\g_2|+3)}
{|\g_1\g_2|} \int_{2^n}^{2^{n+1}}\iiiint\limits_{[0,2\eps]^4} e^{-y(v_1+v_2)}\,
du_j dv_j dy \\
&\ll 2^{-n}.
\end{split}
\ee

%%%%%%%%%    Sigma_{2,3}

For the remaining sums, for brevity we define
$$
\rho_1 = \frac12 + i\g_1, \qquad \rho_2 = \frac12 - i\g_2.
$$
Next,
\begin{multline*}
\int_{2^n}^{2^{n+1}} |\Sigma_{2,3}(y)|^2 dy = \int_{2^{n}}^{2^{n+1}} 
 \sum_{|\g_1|, |\g_2| \leq T} m(\g_1) m(\g_2) e^{i y (\g_1 - \g_2)} \rho_1 \rho_2 \\ \times
 \iiiint \limits_{\substack{[0,2\eps]^4 \\ |u_j-v_j| \leq 2^{-n}}}\!\!\!
 \frac{(e^{-v_1y}-e^{-u_1y})(e^{-v_2 y}-e^{-u_2 y})}
{\prod_{j=1}^2 (u_j-v_j)(\rho_j-v_j)(\rho_j-u_j)}
 dv_j dv_j\, dy.
\end{multline*}
By \eqref{euy}, the integrand in the quadruple integral is 
$\ll y^2 e^{-uy-u_1 y}|\rho_1 \rho_2|^{-2}.$
By Lemma \ref{NTchi}, for a given $\g_1$, there are $\ll \log (|\g_1|+3)$ zeros $\g_2$ with $|\g_1-\g_2| < 1$.  Hence, the contribution from terms with $|\g_1-\g_2| < 1$ is
$$
\ll 2^{-n} \sum_{|\g_1-\g_2| < 1} \frac{m(\g_1)m(\g_2)}{|\rho_1 \rho_2|}
\ll 2^{-n} \sum_{\g_1} \frac{\log^3 (|\g_1|+3)}{|\g_1|^2} \ll 2^{-n}.
$$
Using integration by parts,  we have
$$
\int_{2^n}^{2^{n+1}} e^{iy(\g_1-\g_2)} (e^{-v_1y}-e^{-u_2y})(e^{-v_1y}-e^{-u_2y})\, dy
\ll \frac{2^{3n} |u_1-v_1|\, |u_2-v_2| e^{-2^n(u_1+u_2)}}{|\g_1-\g_2|}
$$
uniformly in $u_1,v_1,u_2,v_2$.  Thus, by \eqref{mgam} and Lemma \ref{sumg1g2},
the contribution from terms with  $|\g_1-\g_2| \ge 1$ is
$$
\ll 2^{-n} \sum_{|\g_1-\g_2| \ge 1} \frac{m(\g_1)m(\g_2)}
{|\rho_1 \rho_2| \cdot |\g_1-\g_2|} \ll 2^{-n}.
$$
Combining these estimates, we have
\begin{equation} \label{Sigma23}
\int_{2^n}^{2^{n+1}} |\Sigma_{2,3}(y)|^2 dy \ll 2^{-n}.
\end{equation}

%%%%%%%%%%%%    Sigma_{2,4}

In the same manner, we have
$$
\int_{2^n}^{2^{n+1}} \!\!\! |\Sigma_{2,4}(y)|^2 dy\ =  \!\!
\sum_{\substack{|\g_1|\le T \\ |\g_2| \leq T}} \!\! m(\g_1)m(\g_2)
\rho_1 \rho_2 \int_{2^{n}}^{2^{n+1}} \!\!\!\!
\iiiint\limits_{\substack{[0,2\eps]^4 \\ u_j \le v_j - 2^{-n}}}
\frac{ e^{y(-v_1-v_2+i(\g_1-\g_2))}du_j dv_j}{\prod_{j=1}^2 
(u_j-v_j)(\rho_j-v_j)(\rho_j-u_j)}  dy.
$$
The contribution to the right side from terms with
 $|\g_1-\g_2|<1$ is
\begin{align*}
&\ll \sum_{|\g_1-\g_2| < 1}
\frac{m(\g_1)m(\g_2)}{|\g_1\g_2|} \int_{2^{n}}^{2^{n+1}} 
\Biggl( \int_{2^{-n}}^{2 \eps} \int_{0}^{v-2^{-n}} \frac{e^{-yv}}{(v-u)}\, du\, dv
\Biggr)^2  \\
&\ll \sum_{\g_1} \frac{\log^3 (|\g_1|+3)} {|\g_1|^2}
\int_{2^{n}}^{2^{n+1}} \Biggl( \int_{1/y}^{\infty} e^{-yv} \log(y v)\, dv \Biggr)^2 
\ll 2^{-n}.
\end{align*}
The terms with $|\g_1-\g_2|>1$ contribute
\begin{align*}
&\ll \sum_{\substack{|\g_1|, |\g_2| < T \\ |\g_1-\g_2| > 1}}  \frac{m(\g_1)m(\g_2)}
{|\g_1\g_2| \cdot |\g_1-\g_2|} \Biggl( \int_{2^{-n}}^{2 \eps} \int_{0}^{v-2^{-n}} 
\frac{ e^{-2^nv}}{v-u}\, du\ dv \Biggr)^2 \\
&\ll \sum_{|\g_1-\g_2| > 1}  \frac{\log (|\g_1|+3) \log (|\g_2|+3)}
{|\g_1\g_2| \cdot |\g_1-\g_2| } \pfrac{1}{2^n}^2 \ll \frac{1}{2^{2n}}.
\end{align*}
Therefore,
\begin{equation} \label{Sigma24}
\int_{2^n}^{2^{n+1}} |\Sigma_{2,4}(y)|^2 dy \ll 2^{-n}.
\end{equation}

%%%%%%%%%%%      Sigma_{2,1}

Estimating an average of $\Sigma_{2,1}(y)$ is more complicated, since 
$R(\g,w)$ could be very large if $|w|$ is small and there is another
 $\g'$ very close to $\g$. 
We get around the problem by noticing that $R(\g,w)+R(\g,-w)$ is always small.
We first have, by \eqref{euy} and Lemma \ref{NTchi},
\begin{multline}\label{Sig211}
 \int_{2^n}^{2^{n+1}} |\Sigma_{2,1}(y)|^2 dy \ll
\sum_{\g_1,\g_2} \log^2(|\g_1|+3) \log^2(|\g_2|+3) 
\max_{\substack{0<|\g_1-\g_1'|\le 1  \\
0<|\g_2-\g_2'|\le 1}} \int_{2^n}^{2^{n+1}} e^{iy(\g_1-\g_2)}  \\
\times
 \iiiint\limits_{[0,2\eps]^4} 
\frac{ e^{y(-v_1-v_2)}}
{(u_1-v_1+i\xi_1)(\rho_1-v_1)(u_2-v_2+i\xi_2)(\rho_2-v_2)}
du_j dv_j\, dy,
\end{multline}
where $\xi_1=\g_1-\g_1'$ and $\xi_2=-(\g_2-\g_2')$.
Let 
$$
M(\g) = \max_{\substack{|\g-\g_1|\leq 1 \\ 0 < |\g_1-\g_1'| < 1}} \frac{2}{|\g_1-\g_1'|}.
$$
By Lemmas \ref{zeta} and \ref{Qint}, 
the terms with $|\g_1-\g_2|<1$ contribute
\begin{align*}
&\ll  \sum_{|\g_1-\g_2| < 1} \frac{ 
\log^2(|\g_1|+3)\log^2(|\g_2|+3)}
{|\g_1\g_2|} \int_{2^n}^{2^{n+1}} \frac{1}{y^2}
  \prod_{j=1}^2  \log     
  \(\min\(2y,\frac{2}{|\g_j - \g_j'|}\)\)  \, dy \\
&\ll \frac{1}{2^{n}} \sum_{\g_1} \frac{\log^5 (|\g_1|+3)}{|\g_1|^2}
\log^2 \( \min(2^{n+2},M(\g)) \) \, = o\pfrac{n^2}{2^n} \qquad (n\to\infty).
\end{align*}

Now suppose $|\g_1-\g_2|>1$.  With $\g_1,\g_2,\g_1',\g_2'$ all fixed, let
$\Delta=\g_1-\g_2$.
Fixing $u_1,v_1,u_2,v_2$, we integrate over $y$ first.  The quintuple
integral in \eqref{Sig211} is $J(2^{n+1})-J(2^n)$, where
$$
J(y)=e^{iy\Delta} \iiiint\limits_{[0,2\eps]^4} \frac{e^{-y(v_1+v_2)}}
{(i\Delta-v_1-v_2)\prod_{j=1}^2 (u_j-v_j+i\xi_j)(\rho_j-v_j)} \, du_j dv_j.
$$
Using
$$
\frac{1}{i\Delta-v_1-v_2}=\frac{1}{i\Delta}\sum_{k=0}^\infty 
\pfrac{v_1+v_2}{i\Delta}^k=\sum_{a,b\ge 0} \binom{a+b}{a} \frac{v_1^a v_2^b}
{(i\Delta)^{a+b}},
$$
together with Lemma \ref{Qint}, yields
$$
|J(y)| \ll \frac{\log^2 y}{|\rho_1 \rho_2 \Delta| y^2} \sum_{a,b\ge 0}
\binom{a+b}{a} \pfrac{4\eps}{|\Delta|}^{a+b} \ll 
 \frac{\log^2 y}{|\rho_1 \rho_2 \Delta| y^2}.
$$
Therefore, by Lemma \ref{sumg1g2},
$$
\sum_{\g_1,\g_2} \log^2(|\g_1|+3) \log^2(|\g_2|+3) 
\max_{\substack{0<|\g_1-\g_1'|\le 1 \\
0<|\g_2-\g_2'|\le 1}} |J(2^{n+1})-J(2^n)| \ll \frac{n^2}{2^{2n}},
$$
and hence
\be\label{Sigma21}
\int_{2^n}^{2^{n+1}} \left| \Sigma_{2,1}(y) \right|^2 = o(n^2 2^{-n}).
\ee
Define
$$
\Sigma_2(x;\chi) = (\log x) \sum_{j=1}^4 \Sigma_{2,j}(\log x).
$$
By \eqref{Sigma22}, \eqref{Sigma23}, \eqref{Sigma24}
and \eqref{Sigma21},
\[
\int_2^Y |\Sigma_2(e^y;\chi)|^2\, dy \le 4 \sum_{j=1}^4 
\sum_{n\le \frac{\log Y}{\log 2} + 1} 2^{2n}
 \int_{2^n}^{2^{n+1}}  |\Sigma_{2,j}(y)|^2\, dy
= o (Y \log^2 Y) \qquad (Y\to\infty).
\]
This completes the proof of Theorem \ref{doubleint}.
\end{proof}

%%%%%%%%%%%%%%%%%%%%%%%%%%%%%%%%%%%%%%%%%%%%%%
%
%
\section{Proof of Lemma \ref{TT0}}\label{sec:TT0}
%
%
%%%%%%%%%%%%%%%%%%%%%%%%%%%%%%%%%%%%%%%%%%%%%%

Put $y=\log x$.  For any $\g$ we have
\begin{align*}
\int_0^{2 \eps-2^{-n}} &\frac{ e^{-yv}}{\frac{1}{2}-v+i \g} \int_{v+2^{-n}}^{2 \eps} \frac{du}{(u-v)(\frac{1}{2}-u+i \g)} dv \\
&= \int_0^{2 \eps-2^{-n}}  e^{-yv} \Big(\frac{1}{\frac{1}{2}+i\g} +
 O(\frac{v}{\frac{1}{4}+{\g}^2}) \Big) 
\int_{v+2^{-n}}^{2 \eps} \Big(\frac{1}{\frac{1}{2}+i\g} + O(\frac{u}{\frac{1}{4}+{\g}^2}) \Big) \frac{du}{u-v} dv \\
&= \frac{M+E}{(1/2+i\g)^2},
\end{align*}
where
$$
M = 
\int_0^{2 \eps-2^{-n}}  e^{-yv} \( \log(2\eps-v)+\log 2^n \)\, dv = 
\frac{\log y + O(1)}{y}
$$
and
\begin{align*}
E &\ll \int_0^{2\eps-2^{-n}} e^{-yv} \int_{v+2^{-n}}^{2\eps}
\frac{u}{u-v}\, du\, dv \\
&\ll \int_0^{2\eps-2^{-n}} e^{-yv} \( 1 + v\log 2^n + v\log (2\eps-v) \)
\, dv \ll \frac{1}{y}.
\end{align*}
Hence, the zeros with $|\g| \le T_0$ contribute
$$
 \frac{2\log\log x}{\log x} \sum_{\substack{|\g|\le T_0 \\ \g \text{ distinct}}}
 \frac{m^2(\g)x^{i\g}}{1/2+i\g} + O\pfrac{\log^3 T_0}{\log x}.
$$

Next, let $\Sigma_3(x;T_0)$ be the sum over zeros with $T_0 < |\g| \le T$.
We have
\begin{multline}\label{lem3.7}
\int_{2^n}^{2^{n+1}} |\Sigma_3(e^y,T_0)|^2 dy \le
 \sum_{T_0 \leq |\g_1|, |\g_2| \leq T}  2^{2n+2} m(\g_1) m(\g_2)
\(\frac{1}{2} + i\g_1\)
 \(\frac{1}{2} - i\g_2\)
\\ \int_{2^n}^{2^{n+1}} e^{y i (\g_1 - \g_2)}
\iiiint\limits_{u_j\ge v_j+2^{-n}}
\frac{e^{-yv_1-yv_2}}{\prod_{j=1}^2 (u_j-v_j)(\frac{1}{2}-v_j+i \g_j)
(\frac{1}{2}-u_j+i \g_j)} du_j dv_j\ dy.
\end{multline}

The sum over $|\g_1-\g_2|<1$ on the right side of \eqref{lem3.7} is
\begin{align*}
&\ll \sum_{\substack{T_0 \leq |\g_1|, |\g_2| \leq T \\ |\g_1-\g_2| < 1}}
 \frac{2^{2n}m(\g_1)m(\g_2)}{|\g_1||\g_2|} \int_{2^n}^{2^{n+1}}
 \iiiint\limits_{u_j\ge v_j+2^{-n}}
 \frac{ e^{-yv_1-yv_2}}{(u_1-v_1)(u_2-v_2)\, }du_j dv_j\,  dy \\
&\ll \sum_{\substack{T_0 \leq |\g_1|, |\g_2| \leq T \\ |\g_1-\g_2| < 1}} 
  \frac{n^22^nm(\g_1)m(\g_2)}{|\g_1||\g_2|} 
\ll n^2 2^n \sum_{|\g| \ge T_0} \frac{\log^3(|\g|+3)}
{|\g|} \ll \frac{n^2 2^n \log^5 T_0}{{T_0}},
\end{align*}
applying Lemma~\ref{NTchi}.
The terms where $|\g_1-\g_2|>1$ on the right hand side of \eqref{lem3.7} total
\begin{align*}
&\ll \sum_{\substack{T_0 \leq |\g_1|, |\g_2| \leq T \\ |\g_1-\g_2| > 1}} 
  \frac{2^{2n}m(\g_1)m(\g_2)}{|\g_1||\g_2||\g_1-\g_2|} 
\iiiint\limits_{u_j\ge v_j+2^{-n}}
  \frac{ e^{-2^nv_1-2^nv_2}}{(u_1-v_1)(u_2-v_2)}du_j d v_j \\
&\ll \sum_{\substack{T_0 \leq |\g_1|, |\g_2| \\ |\g_1-\g_2| > 1}} 
\frac{n^2\log(|\g_1|+3)\log(|\g_2|+3)}{|\g_1||\g_2||\g_1-\g_2|} \ll n^2 
\frac{\log^5 T_0}{T_0}.
\end{align*}
by Lemma \ref{sumg1g2}.  Summing over $n$ proves the lemma.

\end{document}